\documentclass[11pt, letterpaper]{amsart}
\usepackage{sty}

\hypersetup{
pdftitle={A semigroup with linearithmic Dehn function},
pdfsubject={Mathematics, Semigroup Theory, Combinatorial Algebra},
pdfauthor={Roman Repeev},
pdfkeywords={Dehn function, Semigroup}
}

\author[Roman Repeev]{Roman Repeev}
\email{roman.repeev@phystech.edu}
\address{Department of Control and Applied Mathematics, Moscow Institute of Physics and Technology}
\title[A semigroup with linearithmic Dehn function]{A semigroup with linearithmic Dehn function}

\usepackage[backend=biber,style=alphabetic,sorting=nty, giveninits=true, hyperref]{biblatex}
\DeclareNameAlias{author}{family-given}
\usepackage{csquotes}
\begin{document}

\begin{abstract}
    It is known that there is no finitely presented group for which the Dehn function lies asymptotically strictly between linear and quadratic functions. This work presents an example of a semigroup that has Dehn function equivalent to $n \log n$, thus it lies strictly inside the said gap. The example is obtained by symmetrizing the rewriting rules of a particular semi-Thue system, which has the derivational complexity function $n \log n$. We also show that such connection is not universal by providing a semi-Thue system, for which the Dehn function of the symmetrized semigroup asymptotically differs from the derivational complexity of the initial system.
\end{abstract}

\maketitle
\tableofcontents

\section{Preliminaries and introduction}

Consider a set (\emph{alphabet}) $A$, whose elements we will call \emph{symbols}. A \emph{string} or \emph{word} in alphabet $A$ is a finite sequence of symbols from $A$. The \emph{length} of a word $w$ is the number of symbols in $w$, denoted as $|w|$. An \emph{empty word} is a sequence of length 0. The set of all words in alphabet $A$ (including the empty word) is denoted as $A^*$.

\begin{defn}
    A \emph{semi-Thue system} (also called string rewriting system) is a pair $(A, \cc{R})$, where $A$ is a set (alphabet) and $\cc{R}$ is a set of pairs $(L, R)\in A^* \times A^*$ (usually written as $L \ra R$), known as \emph{rewriting rules} or simply \emph{rules}.
    The binary relation $\ra_\cc{R} \subseteq A^* \times A^*$ is defined as follows: $u \ra_\cc{R} v$ if there exist a rule $L \ra R \in \cc{R}$ and some words $\alpha, \beta \in A^*$ such that $u = \alpha L \beta, v = \alpha R \beta$.\\
    We will refer to applying the rules to the word as \emph{transformations}.
\end{defn}

\begin{defn}
Let $(A, \cc{R})$ be a semi-Thue system. Consider a relation $\lra_\cc{R}^*$, which is a reflexive, symmetric and transitive closure of $\ra_\cc{R}$. It is an equivalence relation, and its equivalence classes form a \emph{symmetrized semigroup} $S= Sym(A,\cc{R})$ with an operation $[u] \circ [v] = [uv]$. This semigroup can be  presented with a set of generators $A$ and a set of relations $\cc{R}$. Denote it as $S = \lan\lan A \mid \cc{R} \ran\ran$.\\
If $A$ is finite, then $S$ is called \emph{finitely generated}, if also $\cc{R}$ is finite, then it is called \emph{finitely presented}.\\
If words $u$ and $v$ are equal in $S$, we write $u =_S v$, which is true if and only if $u \lra_\cc{R}^* v$.
\end{defn}

We can also see any semigroup $S = \lan\lan A \mid \cc{R} \ran\ran$ as a rewriting system $(A, \cc{\overline{R}})$, where $\cc{\overline{R}} = \cc{R} \cup \cc{R}^-$, and $\cc{R}^-$ is obtained by symmetrization of rules in $\cc{R}$, meaning that $\cc{R}^- = \left\{ v \ra u \mid u \ra v \in \cc{R} \right\}$.
We will call the applications of the rules from $\cc{R}$ \emph{direct} transformations and denote them as $\overset{+}{\ra}$. Applications of the rules from $\cc{R}^-$ will be called \emph{reverse} transformations and denoted as $\overset{-}{\ra}$.

\smallskip

Since we consider the relation $\lra_\cc{R}^*$, i.e. rules can be applied in both ways, the rules defining a semigroup are often written as $L = R$. In this paper we consider only semigroups which are monoids and use symbol 1 for the empty word. Note that any finitely presented monoid can be obtained naturally as a symmetrized semigroup $Sym(A,\cc{R})$ from the corresponding finite semi-Thue system, and there is freedom of choice for the directions of transformations in $(A,\cc{R})$.


\begin{defn}
    Consider a semigroup $S=\lan\lan A \mid \cc{R} \ran\ran$. The \emph{distance} $d_\cc{R}(u, v)$ between two words $u,v \in A^*$ such that $u \lra_\cc{R}^* v$ is the minimal number of direct and reverse transformations from $\ol{\cc{R}}$ necessary to obtain the word $v$ from  $u$.
\end{defn}

\begin{defn}
    Let $S = \lan\lan A \mid \cc{R} \ran\ran$ be a finitely presented semigroup. Then its \emph{Dehn function} is defined as $D_S: \n{N} \to \n{N} \cup \{0\}$
    \[
    D_S(n) = \max \left\{ d_\cc{R}(u,v) \mid u,v \in A^*,\ u \lra_\cc{R}^* v,\ |u| + |v| \leq n \right\}.
    \]
\end{defn}

The definition of the Dehn function is explicitly dependent on the particular \emph{presentation} of a semigroup. In order to omit such dependency, a special equivalence relation for Dehn functions is defined, which is similar to the $\Theta$-asymptotic for algorithmic complexity.

Recall that a function $f: \n{N} \to \n{N}$ is said to be asymptotically bounded above by a function $g: \n{N} \to \n{N}$ and denoted as $f(n) \in O(g(n))$, if there exist a constant $C > 0$ and a natural $N$ such that for all $n \geq N\ f(n) \leq C g(n)$. A function $f$ is said to be asymptotically bounded below by $g$ (denoted as $f \in \Omega(g(n))$), if $g(n) \in O(f(n))$. If $f \in O(g(n))$ and $f \in \Omega(g(n))$, then it is denoted as $f \in \Theta(g(n))$ (it is said that $f$ and $g$ are \emph{asymptotically equivalent}).

\begin{defn}
    For non-decreasing functions $f, g: \n{N} \to \n{R}_+$, it is said that $f \cleq g$, if there is a constant $c > 0$, such that for every $n$ greater than some $N$
    $$f(n) \leq c g(cn) + cn + c.$$
    $f$ and $g$ are said to be \emph{equivalent}: $f \sim g$, if $f \cleq g$ and $g \cleq f$.\\
    Respectively, if $f \cleq g$, but $g \not\cleq f$, then $f \prec g$.
\end{defn}

If $M_1 = \lan\lan A_1 \mid \cc{R}_1 \ran\ran$ and $M_2 = \lan\lan A_2 \mid \cc{R}_2 \ran\ran$ are two different presentations of the same semigroup, i.e. $M_1 \cong M_2$, then $D_{M_1}(n) \sim D_{M_2}(n)$ (see \cite{madlener1985pseudo, pride1995geometric}). All further discussion will be written considering this equivalence. The standard approach of presenting the group as a semigroup by adding inverse symbols and using them for the relations is fully consistent with the definitions above. Thus, we can speak of Dehn functions of semigroups and groups instead of their presentations. Note that due to the definition of $\cleq$ relation, for any group or semigroup the Dehn function $D(n)$ is at least linear, i.e. $D(n)\cgeq n$.

\bigskip

The Dehn function spectrum for groups has been largely explored in the last three decades. First, Dehn function are closely related to word hyperbolicity, namely the group $G$ is hyperbolic if and only if $D_G(n) \sim n$. Also there is a linear-quadratic gap in the group spectrum: if $D_G(n) \prec n^2$, then $D_G(n) \sim n$. One can see the proof of these facts in \cite{gromov1987hyperbolic, yu1991hyperbolicity, papasoglu1995sub}. It is known that for any $k \in \n{N}$ there is a group with the Dehn function equivalent to $n^k$ (see \cite{baumslag1993isoperimetric}). The work \cite{bridson1999fractional} presented first examples of groups with the Dehn function equivalent to $n^\alpha$, where $\alpha$ is fractional. In the paper \cite{sapir-birget-rips} a broad class of group Dehn functions $\cgeq n^4$ is described. This result was later broadened to the class of functions $\cgeq n^2$ in the work \cite{olshanskii}. These works described the possible Dehn functions in terms of time functions of Turing machines. The overall result was that all superadditive functions which are computed ``quickly enough" are in the group spectrum.

\medskip

There are significantly fewer results regarding the semigroup spectrum. Obviously, one can say that it contains the whole group spectrum, since any group is a semigroup. The opposite, however, might not be true. The work \cite{birget} draws a parallel between embeddings of semigroups into finitely presented semigroups and non-deterministic algorithms. In terms of Dehn functions, the results of this work imply that any (superadditive) function which is asymptotically equal to a square of the time function of some non-deterministic Turing machine lies in the spectrum of semigroup Dehn functions. This provides rather large part of the spectrum for semigroup Dehn functions $\cgeq n^2$. There were several attempts to define hyperbolicity for semigroups, namely geometric approach (see \cite{cain2013hyperbolicity} and word-hyperbolicity (see \cite{gilman2002definition, duncan2004word}). However, the relation of these definitions to the Dehn functions spectra are yet not clear. Also, there were recent attempts to study Dehn functions of semigroups with one defining relation (see \cite{nyberg2022dehn}), which present high interest due to the direct connection with the old problem of decidability for the word problem for these semigroups.

\bigskip

The \emph{goal} of this paper is to present an example of a semigroup with Dehn function equivalent to $n \log n$. Since $n \prec n \log n \prec n^2$, this would show that the semigroup Dehn function spectrum significantly differs from the group one, since there is at least one function inside the linear-quadratic gap equivalent to the Dehn function of a semigroup.

\section{Semi-Thue system $S_0$ with derivational complexity in the class $\Theta(n \log n)$}
\label{S0}


For a semi-Thue system $(A, \cc{R})$, the \emph{derivational depth} is a function $\delta: A^* \to \n{N} \cup \{0, \infty\}$
\[
\delta_{(A, \cc{R})}(W) = \max \left\{ L \mid \text{exists a sequence } W = W_0 \ra_\cc{R} W_1 \ra_\cc{R} \dots \ra_\cc{R} W_L  \right\}.
\]
Namely, the derivational depth of a word $W$ is the maximum length of a transformation sequence starting with $W$. The \emph{derivational complexity} function of the system $\Delta_{(A, \cc{R})}: \n{N} \to \n{N} \cup \{0, \infty\}$ is defined as follows:
\[
\Delta_{(A, \cc{R})}(n) = \max \left\{ \delta_{(A, \cc{R})}(W) \mid |W| = n \right\}.
\]
If $\Delta_{(A, \cc{R})}(n) < \infty$ for all $n \in \n{N}$, then its asymptotic growth can be considered.

\medskip

\subsection{Construction} In the paper \cite{AlTal} a semi-Thue system $S_0 = (A_0, \cc{R}_0)$ is constructed. It is defined by an alphabet
\begin{equation}
\label{eq: alphabet}
\begin{aligned}
    A_0 &= D_0 \cup H_0 \cup \{w\} \quad \text{where} \\
    D_0 &= \{0, 1, \ol{0}, \ol{1}\} \quad H_0 = \{h_0, h_1, h_2, c\}
\end{aligned}
\end{equation}

and a set of rewriting rules $\cc{R}_0$ defined as given below:
\begin{equation}
\label{eq: heads_move}
\begin{aligned}
    h_0 1 &\ra \ol{0} h_1   &h_1 1 &\ra \ol{1} h_2   &h_2 1 &\ra \ol{0} h_1 \\
    h_0 0 &\ra \ol{0} h_0   &h_1 0 &\ra \ol{0} h_1   &h_2 0 &\ra \ol{0} h_2
\end{aligned}
\end{equation}
\begin{equation}
\label{eq: heads_to_c}
    h_1 w \ra cw \quad \quad h_2 w \ra cw
\end{equation}
\begin{equation}
\label{eq: c_move}
    \ol{0}c \ra c0 \quad \quad \ol{1}c \ra c1
\end{equation}
\begin{equation}
\label{eq: c_to_h}
    wc \ra w h_0
\end{equation}

We will refer to elements of $H_0$ as \emph{heads}, in particular, the symbol $c$ will be called \emph{cleaning head}. Elements of $D_0$ will be called \emph{digits} and the symbol $w$ will be called a \emph{wall}. Also, we will use auxiliary sets $D = \{0, 1\}$, whose elements will be called \emph{unmarked digits}, and $\ol{D} = \{\ol{0}, \ol{1}\}$, whose elements will be called \emph{marked digits}.

\medskip

In the paper \cite{AlTal} it is proven that the derivational complexity of this system $\Delta_{S_0} (n)$ lies in the asymptotic class $\Theta (n \log n)$.

\begin{rem}
    Since we talk about asymptotic estimates, the base of the logarithm is not important, however, note that everywhere in this paper by $\log n$ we imply $\log_2 n$.
\end{rem}

\medskip

Note some features of the system. First, both the left and right parts of each rule contain exactly one head symbol. This allows us to talk about transformations as the \emph{movement of heads}. Second, all the left parts of the rules are distinct. This means that if there is only one head symbol in the word, then the transformation is \emph{determined}, i.e. either there are no transformations available or there is only one subword which can be transformed and only one rule that can be applied to it. Moreover, after the transformation there is still only one head, hence the determinacy remains. Another remarkable quality of the system is that it is \emph{length-preserving}: since left and right parts of all rules are of length 2, the length of the word does not change during transformations.

\subsection{Working principle} The system is designed to ``emulate" a Turing machine, where the working head shuttles on the tape and computes the logarithm of a number given in unary form. The working head $h_i$ has three possible memory states: $i=0,1,2$ and moves from left to right, while the auxiliary head $c$ is moving from right to left, all according to the rules described above.

Let us first demonstrate this idea on the key example: consider a word $w h_0 \underbrace{11\dots 1}_{n\text{ times}} w$. Starting here, by the rules \eqref{eq: heads_move} the head moves from left to right, changing its index, in the process replacing each \emph{odd} symbol 1 with symbol 0: the heads $h_0$ and $h_2$ replace 1 with 0 ($1 \to \ol{0}$), while $h_1$ keeps it ($1 \to \ol{1}$). The head symbol changes in the following way: $h_0 \to h_1 \to h_2 \to h_1 \to \dots$, a change happens at each met 1 (each head ``goes through" zeroes without changing). This movement continues until the head meets the right wall. So, in the beginning the transformations carry out like this:
\[
w h_0 1^n w \ra w \ol{0} h_1 1^{n-1} w \ra w \ol{0} \ol{1} h_2 1^{n-2} w \ra^{n-2}
w \ol{0101} \dots \ol{01} h_2 w
\]
(an example for even $n$, if $n$ was odd, then the last symbols would be $\ol{0}h_1w$).

Then, following \eqref{eq: heads_to_c}, the head changes to the cleaning head $c$, which by \eqref{eq: c_move} proceeds to move from right to left, ``unmarking" the digits (replacing $\ol{0}$ with 0 and $\ol{1}$ with 1). After meeting the left wall, by \eqref{eq: c_to_h} the head changes to $h_0$. The previous sequence continues as follows:
\[
w \ol{0101} \dots \ol{01} h_2 w \ra w \ol{0101} \dots \ol{01} c w \ra^n wc01\dots 01w \ra wh_0 01\dots 01w
\]
In such a way, one ``cycle" passes. After that, the transformations continue to go in these cycles, where the head replaces each odd 1 (from the most left) with 0. Thus, on each cycle the number of 1 digits cuts approximately in half, more precisely, if in the beginning of a cycle there were $m$ 1~digits, then after the cycle there would remain $\floor{\frac{m}{2}}$.

This goes on while there still remain 1~digits in the word. After the cycle where the last 1 is ``zeroed" the word looks like $wh_0 \underbrace{00\dots 0}_{n\text{ times}}w$. Then, $h_0$ goes through all zeroes without changing its index, reaching the word $w\underbrace{\ol{00}\dots \ol{0}}_{n\text{ times}}h_0w$, from which there are no possible transformations.

\subsection{Complexity}

The idea of estimating $\Delta_{S_0}(n)$ is that in the main case described above there are $\Theta(\log n)$ cycles, each consisting of $2(n+2)$ steps, therefore the total number of steps is $\Theta(n \log n)$. This is formulated as a lemma:
\begin{lem}
\label{lem: s0 lower}
    For the word $Z = w h_0 1^k w$ takes place
    $$\delta_{S_0}(Z) = k \left( 2 \left( \ceil{\log k} + 1 \right) + 1 \right).$$
\end{lem}
It is proven in the article \cite{AlTal}. It is also shown that the work of the system on the word of the form $wT_1w\dots T_k w$ can be seen as independent work on words $wT_1w,\ wT_2w, \dots, wT_kw$. The difficult part is to consider all the ``bad" cases: more than one head between two walls, ``incorrect" positioning of digits and their marks. In the paper, all cases are considered, and a general lemma is proven:
\begin{lem}
\label{lem: s0 upper}
    For an arbitrary word $Z \in {A_0^*}$ the inequality holds
    \[
    \delta_{S_0}(Z) \leq 2 \left( |Z| + 2 \right)\left( \log |Z| + 2 \right),
    \]
\end{lem}
\noindent which implies an upper estimate $O(n \log n)$ for $\Delta_{S_0}(n)$.

\section{Semigroup with linearithmic Dehn function}

\subsection{Construction and primary observations}

We take the semi-Thue system $S_0$ as a basis and consider a \emph{semigroup} $S = \lan\lan A_0 \mid \cc{R}_0 \ran\ran$. The main goal of this paper is to prove the following theorem.
\begin{thm}
\label{th: main theorem}
    The Dehn function of the semigroup $S = \lan\lan A_0 \mid \cc{R}_0 \ran\ran$, where $A_0$ is defined by \eqref{eq: alphabet} and $\cc{R}_0$ is defined by \eqref{eq: heads_move}, \eqref{eq: heads_to_c}, \eqref{eq: c_move}, \eqref{eq: c_to_h}, is equivalent to $n\log n$: $D_S(n) \sim n \log n$.
\end{thm}

Note that, for a non-decreasing function $f: \n{N} \to \n{N}$, holds
\begin{equation}
\label{eq: dehn equals theta}
f(n) \sim n \log n \iff f(n) \in \Theta(n \log n).
\end{equation}
Therefore we will prove that $D_S(n)$ lies in the class $\Theta(n \log n)$.

\bigskip

In case of the system $S_0$, we also can speak of applying direct and reverse transformations to heads and movement of heads.

The symmetrization makes analysis more difficult. The left parts of the rules in $S$ are no longer distinct: despite this being the case for rules from $\cc{R}_0$, some rules there have coinciding right parts, thus making some rules from $\cc{R}_0^-$ have coinciding left parts. For example, the word $\ol{0}h_1$ can be transformed in three different ways.

Nonetheless, the system is still length-preserving, hence, if $u \lra^* v$, then $|u| = |v|$. Consider the following auxiliary function
\[
\Gamma (n) = \max \left\{ d_\cc{R} (u,v) \mid u,v \in A^*,\ u \lra^*_\cc{R} v,\ |u| = |v| = n \right\}.
\]
Then the Dehn function of $S$ can be presented as $D_S(n) = \max\limits_{1 \leq k \leq \floor{\frac{n}{2}}} \Gamma(k)$. This implies
\begin{equation}
\label{eq: gamma to d}
    \Gamma(n) \in \Theta(n \log n) \Longrightarrow D_S(n) \in \Theta(n \log n).
\end{equation}

With this in mind, we will be looking for an estimate on $\Gamma(n)$, which implies the necessary estimate on $D_S(n)$.

\medskip

By analogy to the derivational depth function, introduce $\gamma: A_0^* \to \n{N}\cup \{0\}$
\[
\gamma (u) = \max \left\{ k \mid \text{there exists a transformation sequence of length $k$, where all words are distinct} \right\}.
\]
We are interested in sequences with distinct words, since in the end we look for the distance between words, i.e. length of the \emph{shortest} sequence, leading from one word to another. If a word occurs twice in the sequence, then the part between these two entries can be omitted. $\gamma$ cannot be infinite (and thus it is defined correctly), since the word length is constant, and there are finitely many distinct words of a certain length.

That being said, if $u \lra^*_\cc{R} v$, then $d_\cc{R}(u,v) \leq \gamma(u)$, hence
\begin{equation}
\label{eq: Gamma}
\Gamma (n) \leq \max \left\{ \gamma(u) \mid |u| = n \right\}.
\end{equation}

\bigskip

Now, let us thoroughly analyse $\gamma$ in different cases.

\subsection{Case of one head}

\subsubsection{General observations}

Consider the words containing exactly one symbol from $H_0$. Denote the set of such words as $U^1$.

In this case the transformations can only be applied to one head. Moreover, the direct transformations are determined, i.e. for each word $u \in U^1$ there is no more than one possible transformation $u \rad v$. Note again that the number of heads remains constant.

\begin{lem}
\label{lem: direct reverse}
If the word $v$ is obtained from the word $u \in U^1$ by a sequence $u~\rar~u'~\rad~v$, then $u = v$.
\end{lem}
\begin{proof}
    Consider the transformation $u~\rar~u'$. Let it be one of the series \eqref{eq: heads_move}: $x h_i \rar h_j y$, where $x \in \ol{D}, y \in D$. The left part of the consequent direct transformation can be only $h_j y$, since it must be applied to the only head (which is $h_j$ in $u'$), and in the direct transformations the heads stand in first position. Such a rule exists --- it is the rule $h_j y \rad x h_i$. It exists since the symmetrical rule exists in $\cc{R}_0'$ (it was applied in the first place). Moreover, there is only one rule with such left part. Thus, the $u' \rad u$ transformation is $h_j y \rad x h_i$. So, the sequence looks like this:
    $$u = w T_1 x h_i T_2 w \rar w T_1 h_j y T_2 w \rad w T_1 x h_i T_2 w = v,$$
    from where one can see that $u = v$.

    The case of first transformation from \eqref{eq: c_move} is similar: the transformation looks like\\ $cx~\rar~yc,\ x~\in~D, y~\in~\ol{D}$. All direct transformations with the head $c$ in the left part have it in second position, therefore the transformation $u' \rad v$ must have the left part $yc$. By analogy to the previous case, we obtain $u = v$.

    For a transformation \eqref{eq: heads_to_c}: $wh_0 \rar wc$, after the transformation $u \rar {u'}$ we have the head $c$ in the word $u$, to the left of $c$ stands a wall $w$, the only possible transformation $u' \rad v$ is $wc \ra wh_0$, which is inverse to the initial one. Therefore, $u = v$.

    The case \eqref{eq: c_to_h} is analogous: the transformation $u \rar u'$ is $cw \rar h_i w$, then the consequent direct transformation has to be $h_i w \rad cw$, therefore $u = v$.
\end{proof}
\begin{cor}
\label{sec: direct reverse}
If a sequence $\alpha$ of transformations $u \ra \dots \ra v$ is a sequence with distinct words, then some initial segment of $\alpha$ contains only direct transformations, and the rest of $\alpha$ contains only reverse transformations.\\
In particular, this holds for the \emph{shortest} sequence, transforming $u$ into $v$.
\end{cor}

Introduce two more auxiliary functions of words:
\begin{enumerate}
    \item $\gamma_+ (u)$ --- length of the longest possible sequence of direct transformations with distinct words starting with $u$
    \item $\gamma_- (u)$ --- length of the longest possible sequence of reverse transformations with distinct words starting with $u$
\end{enumerate}

\begin{lem}
\label{lem: gamma one}
    For a word $u \in U^1$ the inequality holds: $\gamma(u) \leq \gamma_+ (u) + \max\limits_{v \in U^1,\ |v| = |u|} \gamma_- (v)$.
\end{lem}
\begin{proof}
    Consider the longest possible sequence with distinct words starting with $u$. From Corollary~\ref{sec: direct reverse}, it looks like $\underbrace{u \rad \dots \rad }_{l\text{ steps}}x \underbrace{\rar \dots \rar z}_{k\text{ steps}}$.\\
    $l \leq \gamma_+(u)$ by definition of $\gamma_+(u)$.\\
    We cannot be sure, precisely which word is $x$, but we can claim that
    \[
    k \leq \gamma_- (x) \leq \max\limits_{v \in U^1,\ |v| = |u|} \gamma_- (v).
    \]
    Then, $\gamma(u) = l + k \leq \gamma_+(u) + \max\limits_{v \in U^1,\ |v| = |u|} \gamma_- (v)$.
\end{proof}

\bigskip

Let us get a lower estimate on $\Gamma(n)$. Consider a word $u = w h_0 1^k w$, let $k = n - 3$, i.e. $|u| = n$.

\begin{lem}
\label{lem: S lower}
    For the word $u = w h_0 1^k w$ there is a word $v$ such that
    \[
    d(u,v) > 2(k+1)\floor{\log k}.
    \]
\end{lem}
\begin{proof}
    Start a sequence of direct transformations from the word $u$. Direct transformations are determined (for one head), hence this sequence can develop in only one way. Like described in Section \ref{S0}, these transformations go in cycles, and after each cycle the number of ones in the word change from $m$ to $\floor{\frac{m}{2}}$. Each cycle takes $2(k + 1)$ steps: two runs through all the digits take $k$ steps each, and also there are changes into head $c$ and back to $h_0$.

    After $\floor{\log k}$ cycles there will be exactly one symbol 1 left: the word will have a form\\ $wh_0 0^{l-1} 1 0^{k - l} w$. Then, after $l > 0$ steps, the head will go through the symbol 1, turning it to $\ol{0}$, and the word will look like $v =  w \ol{0}^{l} h_1 0^{k-l} w$.

    The sequence described above looks like this (all transformations are direct):
    \begin{multline}
    \label{eq: alpha sequence}
        u = w h_0 1^k w \underbrace{\ra^{2(k+1)}}_{\text{one cycle}} w h_0 (01)^{\frac{k}{2}} w \ra \overset{\floor{\log k} - 1 \text{ cycles}}{\cdots} \ra \\ w h_0 0^{l - 1} 1 0^{k - l} w \ra^{l-1} 
         w \ol{0}^{l-1} h_0 1 0^{k-l} w  \ra w \ol{0}^l h_1 0^{k-l} w = v.
    \end{multline}

    The outline above is an example for even $k$. In case of odd $k$ the word after the first cycle would be $w h_0 (01)^{\floor{\frac{k}{2}}}0 w$.

    Name the sequence \eqref{eq: alpha sequence} as sequence $\alpha$. Its length equals $2(k + 1)\floor{\log k} + l > 2(k + 1)\floor{\log k}$. Let us prove that this is the shortest sequence, transforming $u$ into $v$.

    \medskip

    Notice that reverse transformations \emph{do not reduce} the number of 1 digits in the word. Consider the \emph{shortest} sequence (name it sequence $\beta$):
    $$u = W_0 \ra W_1 \ra \dots \ra W_{m - 1} \ra W_m = v.$$
    Since there are some 1 digits in $v$ and no in $v$, the number of them decreases, therefore $\beta$ must contain at least one direct transformation. Since $\beta$ is the shortest sequence, then, by Corollary \ref{sec: direct reverse}, all direct transformations come in the beginning of $\beta$. Since direct transformations are determined, some initial part of $\beta$ coincides with some initial part of $\alpha$. Notice that in $\alpha$ all words besides the last one ($v$) contain at least one 1. This means that if we stop $\alpha$ before its end and start reverse transformations (which do not decrease the number of 1 digits), then the number of 1 digits will never be equal to 0, therefore the word $v$ will be impossible to reach. This means that the initial part of $\beta$ coincides with the whole sequence $\alpha$, therefore the whole $\beta$ coincides with $\alpha$, since it is the shortest one.

    Thus, $\alpha$ is the shortest sequence transforming $u$ into $v$, meaning that its length equals $d(u,v)$. Therefore,
    \[
    d(u,v) = 2(k+1)\floor{\log k} + l > 2(k + 1)\floor{\log k}.
    \]
\end{proof}
\begin{cor}
\label{cor: Gamma lower}
    $\Gamma(n) \in \Omega(n \log n)$.
\end{cor}

\subsubsection{Main case}

Now consider words of the form
\begin{equation}
\label{eq: main case}
    u = w T_1 h_i T_2 w \quad \text{ where}\quad T_1 \in \ol{D}^*,\ T_2 \in D^*,\ h \in H_0.
\end{equation}
We look for the longest possible sequence of \emph{reverse} transformations with distinct words, thus getting an upper estimate on $\gamma_-(u)$ for words $u$ of the form \eqref{eq: main case}.

\medskip

First, notice that direct transformations leave a marked digit to the left of the head, and reverse ones leave an unmarked digit to the right of the head. We will say that the words of the form \eqref{eq: main case} have the \emph{correct marking of digits}, meaning that all the digits to the left of the head are marked and all the digits to the right are unmarked. Then in an arbitrary sequence of transformations starting with a word of the form \eqref{eq: main case} all the words have the correct marking of digits. Keeping this in mind, we will partly omit mentioning the changing of marking during transformations.

Consider what each head does during reverse transformations.

The simplest one is the cleaning head $c$. Its movement is determined: by \eqref{eq: c_move}, it moves left to right, marking the digits (changing $0$ into $\ol{0}$ and $1$ into $\ol{1}$) until reaching the right wall $w$. Upon reaching the wall, by \eqref{eq: c_to_h}, it transforms into either head $h_1$ or $h_2$. We will address this ambiguity later.

Now consider the behaviour of heads $h_i, i\in\{0,1,2\}$. All these heads move through digits right to left. We will talk about their movement as moving through \emph{blocks} of zeroes. A \emph{block} is a subword of one of the forms $x_l \ol{0}\dots \ol{0} h_i 0 \dots 0 x_r$, $x_l \ol{0}\dots \ol{0} 0 \dots 0 x_r$ or $x_l 0 \dots 0 x_r$, where $x_l$ and $x_r$ are either walls or 1 digits (marked or unmarked --- this is determined by the correct marking of digits in the word). We will refer to $x_l$ and $x_r$ as left and right \emph{edges} of the block. Then the word can be divided into these blocks (with overlapping edges). We will also say that the head meets a \emph{dead end} if there are no reverse transformations available for it.
\begin{enumerate}
    \item Head $h_0$. One can see that the reverse transformations of head $h_0$ are determined. It moves through zeroes in a block until it reaches its left edge. If the edge is $\ol{1}$, then it is a dead end. If it is a wall, then a transformation $w h_0 \rar w c$ is available, after it the sequence may continue with the cleaning head (described above).
    \item Head $h_2$. Its behaviour is also determined and is similar to the one of the head $h_0$. It moves through zeroes until reaching the left edge of the block. If the edge is a wall, then it is a dead end. If it is $\ol{1}$, then a transformation $\ol{1} h_2 \rar h_1 1$ is available. After this transformation, the head changes into $h_1$ and moves into the next block.
    \item Head $h_1$. This head may also move through zeroes. However, if it reaches the left edge, then it meets a dead end no matter which one it is. Another option is one of the transformations $\ol{0} h_1 \rar h_2 1$ or $\ol{0} h_1 \rar h_0 1$. After such transformation, the head changes into $h_0$ or $h_2$, while also turning a 0 digit into a 1 digit, which will be called \emph{producing} a 1 digit. Notice that the choice of the head into which $h_1$ should turn is determined by our desire to continue the sequence as long as we can: if the left edge of the current block is $w$, then $h_1$ should be turned into $h_0$, if it is $\ol{1}$, then it should be turned into $h_2$ (otherwise a dead end will be met at the left edge).
\end{enumerate}

\begin{rem}
\label{rem: 01}
    Notice that a head can move through a block if and only if there is at least one 0 digit inside the block.
\end{rem}

\begin{lem}
\label{lem: main case upper}
    If a word $u$ with length $k + 3$ is of form \eqref{eq: main case}, then
    $$\gamma_-(u) \leq 2(k+1)\ceil{\log k} + 6k + 4 .$$
\end{lem}
\begin{proof}
Let us find the longest possible sequence of reverse transformations starting with $u$.

\noindent\textbf{Part 1.} Consider a possible part of the sequence, where there are no 1 digits in the word and no 1 digits are being produced. Generally, the word is of form $w\ol{0}^l h 0^{k-l} w,\ h \in H_0,\ l \geq 0$.
\begin{enumerate}
    \item $h = h_0$. Determined part: $w\ol{0}^l h_0 0^{k-l} w \ra^l w h_0 0^k w \ra w c 0^k w$.\\
    Further path is determined by the case $h = c$.
    \item $h = c$. Determined part: $w\ol{0}^l c 0^{k-l} w \ra^{k-l} w \ol{0}^k c w \ra w \ol{0}^k h_i$,
    where $i \in \{1, 2\}$.\\
    Further path is determined by the cases $h = h_1$ and $h = h_2$.
    \item $h = h_2$. Determined part: $w\ol{0}^l h_2 0^{k-l} w \ra^l w h_2 0^k w$ --- a dead end.
    \item $h = h_1$. Either the path goes like w $\ol{0}^l h_1 0^{k-l} w \ra^l wh_1 0^k w$ --- a dead end, or somewhere a 1 digit is produced, which takes us to Part 2.
\end{enumerate}
Then, the longest possible Part 1 of the sequence, which goes to Part 2 (this allows to further continue the sequence and not just meet a dead end), looks like
\[
w\ol{0}^k h_0 w \ra^k wh_0 0^k w \ra wc 0^k w \ra^k w\ol{0}^k c w\ra w \ol{0}^k h_1 w.
\]
Its length equals $2(k+1)$ steps.

If the starting word $u$ is not $w\ol{0}^k h_0 w$, then it is somewhere in the sequence above or the sequence would definitely meet a dead end, not reaching Part 2. In both cases the Part 1 is not longer than $3k + 2$.

\medskip

\noindent\textbf{Part 2.} This is the part where 1 digits start being produced during the movement of $h$ head to the left. Consider the word $v = w \ol{0}^k h_1 w$, on which the Part 1 finished --- this will be our main case. The head $h_1$ will move right to left, producing a single 1 digit somewhere in the word. Then, it will turn into a $c$ head, which will move through the digits to the right wall, turning into a head $h_1$ or $h_2$. Both options are possible.

Further, the head will continue moving in such cycles: from the right wall to the left wall, while producing a 1 digit inside blocks to the left of each already existing 1 digits in the word. Optionally, a 1 digit inside the rightmost block can also be produced (this depends on which of the heads $h_1$ and $h_2$ is the head in the beginning of the cycle). After reaching the left wall, the head turns into $c$, moving to the right wall and turning back into $h_1$ or $h_2$, thus ending the cycle and beginning a new one.

The length of each cycle equals $2(k+1)$ steps, and after each cycle the number of 1 digits is at least doubled (or it turns from $m$ to $2m + 1$).

The Remark \ref{rem: 01} implies that, in order for the head $h$ to move fully from the right wall to the left, the number of 1 digits in the word at the beginning of the cycle must be no more that $\frac{k}{2}$. After the first cycle this number will equal 1. Then, after another $\ceil{\log k} - 1$ cycles (if such a sequence is possible) the number of 1 digits would be $\geq 2^{\ceil{\log k} - 1} \geq \frac{k}{2}$. This means that after that no more than one full cycle is possible (it is possible if the number of digits equals $\frac{k}{2}$, i.e. $k$ is a power of 2). After this possible cycle the head will not be able to fully go from the right wall to the left, thus making $< k$ steps before meeting a dead end.

A rough estimate on the length of the Part 2 is such: not more than $\ceil{\log k} + 1$ cycles, each takes $2(k+1)$ steps, then less than $k$ steps until a dead end. Overall estimate of the whole sequence length then is such:
\[
\gamma_-(u) \leq \underbrace{3k + 2}_{\text{Part 1}} + \underbrace{2(k+1)(\ceil{\log k}+1) + k}_{\text{Part 2}} = 2(k+1)\ceil{\log k} + 6k + 4.
\]
\end{proof}

\subsubsection{Case of one or no walls}

Consider a word, where there is less than two $w$ symbols.

\begin{enumerate}
    \item \textbf{No walls.} A word of form $u = T_1xT_2$, where $T_1, T_2 \in D_0,\ x \in H_0$.

    In this case, no matter which head is $x$, direct transformations will move $x$ in one direction and reverse transformations --- in another (since there are no walls, the heads $h_i$ and $c$ cannot turn into one another). Then the longest possible sequence with distinct words (first direct transformations, then reverse ones) at best will move the head from one end of the word to another, and then back to the initial end, then its length is not greater than $2|u|$.

    \item \textbf{One left wall.} A word of form $u = wT_1xT_2$, where $T_1, T_2 \in D_0,\ x \in H_0$.

    If $x \neq c$, then the longest sequence with distinct words is such: $x$ moves to the left wall by reverse transformations, then it turns into $c$, and moves to the right end by reverse transformations (this is a dead end, since only direct transformations are available, and we forbade direct transformations after reverse ones). Another option is to start the sequence with direct transformations, which move the head to the right, but it ends after reaching the right end, since after that only reverse transformations are possible (which are banned). Overall, the head at maximum goes through the word twice, thus the length of the sequence is not greater than $2|u|$

    If $x = c$, then conversely: $x$ moves to the left by direct transformations, turns into $h_0$, then moves to the right end. Then a sequence of reverse transformations is possible, moving the head to the left wall and then to the right end. Overall length of the sequence is not greater than $4|u|$.

    \item \textbf{One right wall.} A word of form $u = T_1xT_2w$, where $T_1, T_2 \in D_0,\ x \in H_0$.

    By complete analogy with the previous case, the sequence is not longer than $4|u|$.
\end{enumerate}

Let us summarize this into one lemma.
\begin{lem}
\label{lem: upper one wall}
    If a word $u$ is of form $T_1xT_2,\ wT_1xT_2$ or $T_1xT_2w$, where $T_1, T_2 \in D_0^*,\ x \in H_0$,\\ then $\gamma(u) \leq 4|u|$.
\end{lem}

\subsubsection{Summing up}

\begin{lem}
\label{lem: direct upper}
    For an arbitrary $u \in A_0^*$ the inequality holds
    $$\gamma_+(u) \leq 2 \left( |u| + 2 \right)\left( \log |u| + 2 \right).$$
\end{lem}
\begin{proof}
    The maximum length of a sequence of direct transformations with distinct words is not greater than the maximum length of an arbitrary sequence of direct transformations, thus $\gamma_+(u) \leq \delta_{S_0}(u)$. The Lemma \ref{lem: s0 upper} then implies the sought estimate.
\end{proof}

\begin{lem}
\label{lem: one head bad}
    If a word $u \in U^1$ is of form $wT_1 x T_2 w,\ x \in H_0,\ T_1, T_2 \in D_0^*$, however $T_1$ contains a digit from $D$ or $T_2$ contains a digit from $\ol{D}$, then $\gamma(u) \leq 4 |u|$.
\end{lem}
\begin{proof}
    Notice that all transformations, both direct and reverse, have a form $\ol{d}x \ra x' d'$ or $xd \ra \ol{d}' x'$, where $x,x' \in H_0,\ d,d'\in D,\ \ol{d}, \ol{d}' \in\ol{D}$.

    This means that, while moving to the left, a head unmarks a digit, and while moving to the right, it marks one. If there is an unmarked digit straight to the left of a head, then it cannot move to the left, and if there is a marked digit straight to the right, then movement to the right is impossible.

    Consider the case where there is an unmarked digit $d \in D$ in $T_1$.\\
    Let $x = h_i$. Then direct transformations move $x$ to the right, and if it manages to reach the right wall, then in changes into $c$ and starts moving to the left, but it cannot reach the left wall, since it must bump into $d$ (if it does not stop earlier). As a result, the head does not more than a back and forth trip in the word, thus making not more than $2|u|$ steps. The reverse transformations after that move the head in reverse directions, until at best it bumps into $d$ again --- again, not more than $2|u|$ steps. Overall, $\gamma(u) \leq 4|u|$.

    Let $x = c$. Direct transformations move the head to the left until it meets an unmarked digit ($d$ or some before it) --- this path takes not more than $|u|$ steps. The next sequence of reverse transformations at best moves the head to the right wall, and then to the left until it bumps into $d$ --- not more than $2|u|$ steps. Overall, $\gamma(u) \leq 3|u| \leq 4|u|$.

    The case of a marked digit inside $T_2$ is considered by complete analogy.
\end{proof}

\begin{lem}
\label{lem: one head full upper}
    If $u \in U^1$, then $\gamma(u) \leq 4|u| \log |u| + 10|u|$.
\end{lem}
\begin{proof}
    In general, a word with one head can have one of the forms (here, $T_i \in D_0^*, x \in H_0$):

\begin{enumerate}
    \item $u = T_1 x T_2$. By Lemma \ref{lem: upper one wall}, $\gamma(u) \leq 4|u|$.
    \item $u = T_1 w T_2 w \dots w T_k x T_{k+1}$.\\
        Since $x$ cannot go through the wall, $\gamma(u) = \gamma(w T_k x T_{k+1})$. Then, by Lemma \ref{lem: upper one wall},\\ $\gamma(u) \leq 4|w T_k x T_{k+1}| \leq 4|u|$.
    \item $u = T_1 x T_2 w T_3 w \dots w T_k$. The case is analogous to the previous one.
    \item $u = T_1 w T_2 w \dots w T_k x T_{k+1} w T_{k+2} w \dots w T_m$.\\
        Again, since $x$ cannot go through walls, $\gamma(u) = \gamma(v)$, where $v = w T_k x T_{k+1} w$.\\
        If the word $v$ is ``bad", i.e. it satisfies the conditions of Lemma \ref{lem: one head bad}, then, by this lemma, $\gamma(v) \leq 4|v|$.\\
        If $v$ is ``good", i.e. it has a form \eqref{eq: main case}, then after any transformations it remains ``good", so, by analogy to Lemma \ref{lem: gamma one},
        $\gamma(v) \leq \gamma_+(v) + \max\limits_{z \in G} \gamma_-(z)$,\\ where $G$ is a set of ``good" words with length $|v|$.\\
        The Lemma \ref{lem: direct upper} implies
        \[
        \gamma_+(v) \leq 2(|v| + 2)(\log |v| + 2) = 2(|v| + 2)\log |v| + 4|v| + 8,
        \]
        and from Lemma \ref{lem: main case upper} follows
        \[
        \max\limits_{z \in G} \gamma_-(z) \leq 2(|v| - 2)\floor{\log (|v| - 2)} + 6|v| - 14 \leq
        2(|v| - 2)\log |v| + 6|v| - 14.
        \]
        Then,
        \begin{multline*}
        \gamma(v) \leq 2(|v| + 2)\log |v| + 4|v| + 8 + 2(|v| - 2)\log |v| + 6|v| - 14 = \\
        = 4|v| \log |v| + 10|v| - 6 \leq 4|v| \log |v| + 10|v|.
        \end{multline*}
        Considering the fact that $\gamma(u) = \gamma(v)$, and also $|v| \leq |u|$, we obtain
        \[
        \gamma(u) \leq 4|u| \log |u| + 10|u|.
        \]
\end{enumerate}
    Notice that $4|u| < 4|u| \log |u| + 10|u|$, so in all cases the sought estimate has been obtained.
\end{proof}

\subsection{Case of multiple heads}

Consider the words of the form
\begin{equation}
\label{eq: many heads}
    Z = w T_1 x_1 T_2 \dots T_t x_t T_{t+1} w \qquad T_i \in D_0^*,\ x_i \in H_0,\ t \geq 2
\end{equation}

In other words, there is a word with digits and two or more heads between two walls.

Consider an arbitrary sequence of transformations (possibly, infinite):
\begin{equation}
\label{eq: w chain}
    Z = W_0 \ra W_1 \ra \dots
\end{equation}

Since each rule contains a single head in both left and right parts, we can speak of movement of heads and, consequently, the \emph{positions} of each head. Denote as $p_i(s)$ the position of the head $x_i$ in the word $W_s$. Also denote as $W_k[j]$ the symbol in position $j$ in the word $W_k$.

\begin{lem}
\label{lem: position order}
    In the sequence \eqref{eq: w chain} for all $k$ holds $p_1(k) < p_2(k) < \dots < p_t(k)$.
\end{lem}
\begin{proof}
    Induction on $k$.

    Obviously, the proposition holds for $k = 0$. Then, let it be true for an arbitrary $k$. Then, at step $W_k \ra W_{k+1}$ a transformation is applied to some head $x_i$. During this transformation, the position of $x_i$ either does not change (rules with left part $wx_i$ or $x_iw$) or changes by $\pm 1$, in the second case the symbol $W_{k}[p_i(k+1)]$ (it is the symbol with which the head $x_i$ engages) lies in $D_0$. In all cases the proposition holds.
\end{proof}

For a sequence $W_0 \ra \dots \ra W_k$, name a \emph{coverage} of the head $x_i$ the set
\begin{equation}
\label{eq: cov def}
    \rmb{c}_i(k) = \left\{ n \mid \exists m \leq k: p_i(m) = n \right\},
\end{equation}
i.e. the set of positions, in which the head $x_i$ has been during the sequence $W_0 \ra \dots \ra W_k$.

Since for each transformation the position of a head modifies at most by 1, the coverage of any head is a closed interval $[a,b]$ (here we consider intervals as sets of natural numbers).

We will write $\rmb{c}_i < \rmb{c}_j$ in the sense that $\rmb{c}_i = [m_1,n_1],\ \rmb{c}_j = [m_2,n_2]$ and $n_1 < m_2$.\\
Also, if $\alpha$ is an interval $[a,b]$, then by $W[\alpha]$ we will mean the subword $W[a]W[a+1]\dots W[b]$.

\begin{lem}
\label{lem: cov}
    In any sequence \eqref{eq: w chain} starting with a word of the form \eqref{eq: many heads} for all $k$ holds
    $$(*)\quad \rmb{c}_s(k) < \rmb{c}_{s+1}(k)\quad \text{for all } 1 \leq s \leq t - 1 $$
    moreover, for all $s$ 
    $$(**) \quad W_k[\rmb{c}_s(k)] \text{ has a form } P_1 x_s P_2,\quad \text{ where } P_1 \in \ol{D}^*,\ P_2 \in D^*.$$
\end{lem}
\begin{proof}
    We prove this by induction on $k$.

    For $k = 0$ by direct observation one can assure that for all $s$ one has $\rmb{c}_s(0) = [p_s(0), p_s(0)]$, and all statements hold.
    Let all statements be true for some arbitrary $k$.

    \smallskip

    Consider the transformation $W_k \ra W_{k+1}$. It is applied to some particular head $x_i$.

    As mentioned earlier, with each transformation the head unmatches a digit, while moving to the left, and matches a digit, while moving to the right, i.e. each transformation looks like
    
    $$xd \ra \ol{d} x' \quad \text{or} \quad dx \ra x'\ol{d} \quad \text{where } d\in D,\ \ol{d} \in \ol{D}.$$
    
    Consider all possible cases of transformation.

        \textbf{Case 1.} A transformation with wall: $wx \ra wx'$ or $xw \ra x'w$.\\
        Here the position of the head and all other symbols do not change, hence the inequalities for coverages and the structure of the coverages do not change, so, since all statements were true for $W_k$, they remain true for $W_{k+1}$.
        
        \textbf{Case 2.} Movement of the head to the left: $\ol{d} x_i \ra x_i' d$.\\
        Let $\rmb{c}_i(k)=[a,b]$. Consider different cases of location of the head inside its coverage.
        \begin{enumerate}[label=\alph*)]
            \item $a < p_i(k) \leq b$. Then, according to the induction assumption,
            $$W_k[\rmb{c}_i(k)] = P_1 \ol{d} x_i P_2, \text{ for some } P_1 \in \ol{D}^*, P_2 \in D^*.$$
            During $W_k \ra W_{k+1}$ the position of $x_i$ decreases by 1: $p_i(k+1) = p_i(k) - 1$, then $a \leq p_i(k+1) < b$, hence $\rmb{c}_i(k+1) = [a,b] = \rmb{c}_i(k)$.

            For $j \neq i$ the coverages do not change: $\rmb{c}_j(k+1) = \rmb{c}_j(k)$. So, for all $l$ one has $\rmb{c}_l(k + 1) = \rmb{c}_l(k)$, hence $(*)$ for $k+1$ is equivalent to $(*)$ for $k$, which is true by induction assumption.

            Moreover, $W_{k+1}[\rmb{c}_j(k+1)] = W_k[\rmb{c}_j(k)]$ for $j \neq i$, since these subwords are not affected during the transformation, so one has $(**)$ for $s \neq i$ and $k+1$.

            For $x_i$ the transformation is such:
            $$W_k[\rmb{c}_i(k)] = P_1 \ol{d} x_i P_2 \ra P_1 x_i' d P_2 = W_{k+1}[\rmb{c}_i(k+1)]$$
            so, $W_{k+1}[\rmb{c}_i(k+1)]$ also has the form required by $(**)$.\\
            This means that $(**)$ for $k + 1$ holds for all $s$.

            \item $a = p_i(k) \leq b$. Since the transformation looks like $\ol{d} x_i \ra x_i' d$, one has $W_k[a-1] = \ol{d} \in \ol{D}$. Hence, $a-1$ cannot lie in a different coverage to the left: either $i = 1$, then there is no other coverage to the left; or $\rmb{c}_{i-1}(k) = [p,q]$ --- in this case, by induction assumption, $(**)$ for $s = i - 1$ implies, that $W_k[q] \in D \cup H_0$. This means that $W_k[a - 1] \neq W_k[q] \Ra q \neq a - 1$. Also, $(*)$ implies, that $q \leq a - 1$, so, in the end we get $q < a - 1$.

            This means that neither the coverages nor the subwords covered by them, are affected for $s \neq i$. Hence, $(**)$ for $k+1$ holds for $s \neq i$.

            For $x_i$ we have $\rmb{c}_i(k+1) = [a - 1, b]$. The right end of the coverage does not change, hence $\rmb{c}_i(k+1) < \rmb{c}_{i+1}(k+1)$ (if $i < t$, otherwise there is no $\rmb{c}_{i+1}$), and for the left one we have proven that $\rmb{c}_{i-1}(k+1) < \rmb{c}_i(k+1)$ (again, if $i > 1$). This means, that $(*)$ holds for $k+1$.

            The transformation itself looks like
            $$\ol{d} W_k[\rmb{c}_i(k)] = \ol{d} x_i P \ra x_i' d P = W_{k+1}[\rmb{c}_i(k+1)], \quad P\in D^*,$$
            therefore, $(**)$ for $k+1$ and $s = i$ also holds.
        \end{enumerate}
        
    \textbf{Case 3.} Movement of the head to the right: $x_i d \ra \ol{d} x_i'$. This case is considered by complete analogy to the previous one.
    
    Thus, the induction step is proven, hence the whole statement of the lemma.
\end{proof}

\medskip

Name the \emph{complete coverage} of the head $x_i$ and denote it as $\rmb{C}_i$ the set of positions of $x_i$, that can be reached in some sequence, namely
$$m \in \rmb{C}_i \overset{\mathrm{def}}{\iff} \text{ there is a sequence } Z = W_0 \ra W_1 \ra \dots \ra W_k \text{ such that } p_i(k) = m.$$

\begin{lem}
\label{lem: Cov ind}
    For any word $Z$ of the form \eqref{eq: many heads} for all $i \in \ol{1,t}$ the complete coverages $\rmb{C}_i$ are closed intervals, and $\rmb{C}_1 < \rmb{C}_2 < \dots < \rmb{C}_t$.
\end{lem}
\begin{proof}
    Consider a set of all words that are reachable from $Z$:
\[
F(Z) = \left\{ V \mid Z \ra^* V \right\}.
\]
Since the transformations do not change the length of the word, $|F(Z)|$ is not greater than the number of all words of length $|Z|$, consequently, $|F(Z)| < \infty$.

Enumerate all words in $F(Z)$, i.e. $F(Z) = \left\{ F_l\right\}_{l = 1}^{N}$, where $N = |F(Z)| < \infty$.
If there is a sequence $Z \ra \dots \ra V$, then there is a reverse one $V \ra \dots \ra Z$. Then there is one finite sequence
\[
Z \ra \dots \ra F_1 \ra \dots \ra Z \ra \dots \ra F_2 \ra \dots \ra F_N.
\]

Let this sequence consist of $M$ steps. The set of words in this sequence coincides with $F(Z)$. Then, for all $i$ holds $\rmb{C}_i = \rmb{c}_i(M)$. By using Lemma \ref{lem: cov} with $n = M$ we finish the proof.
\end{proof}

\begin{lem}
\label{lem: Cov digits}
    For any $Z$ of the form \eqref{eq: many heads} for all $i \in \ol{1,t}$ holds $Z[\rmb{C}_i] = V_1 x_i V_2$, where\\ $V_1 \in \ol{D}^*,\ V_2 \in D^*$, wherein it is the longest subword of such form inside $Z$ (meaning that there is no marked digit straight to the left of it, and no unmarked digit straight to the right).
\end{lem}
\begin{proof}
    While proving Lemma \ref{lem: cov} we mentioned that if during a transformation with the head $x_i$ a new position appears in $\rmb{c}_i$, then, firstly, this position is not included in any other coverage on this step, and secondly, there is an unmarked digit on this position if it stands to the right of $x_i$, and a marked digit if it stands to the left. The first fact means that the digit has not yet been changed since the beginning of the sequence, hence the same digit stands in this position in $W_0$.\\
    Moreover, if a marked digit $\ol{d}$ stands straight to the left of $\rmb{c}_i$, then at some point there was a subword $\ol{d}x_i$ in the word, for which a transformation $\ol{d} x_i \ra x_i'd$ is possible, meaning that $\ol{d}$ (more precisely, its position) will be in $\rmb{C}_i$. Similar reasoning applies for an unmarked digit $d$ straight to the right of $\rmb{c}_i$.
\end{proof}

\begin{lem}
\label{lem: comm}
    For any word $Z$ of the form \eqref{eq: many heads} and any finite sequence of transformations
    \[
    Z = W_0 \ra W_1 \ra \dots \ra W_n
    \]
    exists another sequence
    \[
    Z = W_0' \ra W_1' \ra \dots W_n' = W_n
    \]
    in which for each pair of consequent transformations $W_{k}' \ra_1 W_{k+1}' \ra_2 W_{k+2}'$, where $\ra_1$ is applied to the head $x_i$, and $\ra_2$ is applied to $x_j$, holds $i \leq j$.
\end{lem}
\begin{proof}
    Lemmata \ref{lem: Cov ind} and \ref{lem: Cov digits} imply, that the heads act independently of each other: their area of effect (complete coverages) do not intersect, and a head cannot impact the possible transformations, associated with other heads. Therefore, the transformations, associated with different heads, \emph{commute}. Thus, we can rearrange the order of transformations to satisfy the statement of the lemma.
\end{proof}

So, taking the obtained propositions into account, the word $Z$ of the form \eqref{eq: many heads} can be written as
\begin{equation}
\label{eq: many heads 2}
    Z = w u_1 \underbrace{V_1^L x_1 V_1^R}_{\rmb{C}_1} u_2 \underbrace{V_2^L x_2 V_2^R}_{\rmb{C}_2} u_3 \dots u_t \underbrace{V_t^L x_t V_t^R}_{\rmb{C}_t} u_{t+1} w
\end{equation}
where $u_i \in D_0^*$, $V_i^L \in \ol{D}^*$, $V_i^R \in D^*$, $x_i \in H_0$, $t \geq 2$.\\
In addition, the most right symbol of $u_1$ is not from $\ol{D}$, the most left symbol of $u_{t+1}$ is not from $D$, and for other $u_i$ both conditions are met. In particular, this is true for empty $u_i$.

Now we are ready to prove a general estimate for this case.
\begin{lem}
\label{lem: many heads upper linear}
    For any word $Z$ of the form \eqref{eq: many heads} $\gamma(Z) \leq 4|Z|$.
\end{lem}
\begin{proof}
    Consider the longest possible sequence of transformations with distinct words starting from $Z$. Lemma \ref{lem: comm} implies that there is a sequence of same length, where first all the transformations associated with the head $x_1$ occur, then with $x_2$ etc. Then,
    \[
    \gamma(Z) = \gamma^{(1)}(Z) + \gamma^{(2)}(Z) + \dots + \gamma^{(t)}(Z)
    \]
    where $\gamma^{(i)}(Z)$ is the maximum length of a sequence of transformations with distinct word starting from $x_i$.
    Consider a presentation of $Z$ in the form \eqref{eq: many heads 2}. One can see from it that
    \begin{itemize}
        \item $\gamma^{(1)}(Z) \leq \gamma(w V_1^L x_1 V_1^R)$ --- equality holds if $u_1$ is empty, otherwise $\gamma^{(1)}(Z) = \gamma(V_1^L x_1 V_1^R)$.
        \item $\gamma^{t}(Z) \leq \gamma(V_t^L x_t V_t^R w)$ --- similar case.
        \item For $i \not\in \{1, t\}$, $\gamma^{(i)}(Z) = \gamma(V_i^L x_i V_i^R)$.
    \end{itemize}
    Then,
    \[
    \gamma(Z) \leq \gamma(w V_1^L x_1 V_1^R) + \sum\limits_{i\not\in\{1,t\}} \gamma(V_i^L x_i V_i^R) + \gamma(V_t^L x_t V_t^R w).
    \]
    Arguments of each $\gamma$ satisfy the condition of Lemma \ref{lem: upper one wall}, then, by the lemma:
    \[
    \gamma(Z) \leq 4 \Big( |w V_1^L x_1 V_1^R| + |V_2^L x_2 V_2^R| + \dots + |V_t^L x_t V_t^R w| \Big) \leq 4 |Z|.
    \]
\end{proof}

\begin{rem}
\label{rem: many heads}
    Notice that all the reasoning in this section also applies to the words of the forms $wT, Tw$ and $T$, where $T\in A_0 \backslash \{w\}$ and $T$ contains at least 2 head symbols. Therefore, for such words $Z$, the estimate $\gamma(Z) \leq 4|Z|$ also holds.
\end{rem}

\subsection{General case}

This summary is similar to the consideration of general case in the article \cite{AlTal}.

Let $Z$ have at least one $w$ symbol. Then, in general case, the word $Z$ can be written as
\begin{equation}
\label{eq: general case}
    Z = T_1 w T_2 w \dots w T_k \qquad k \geq 2,\ T_i \in \left( A_0 \backslash \{w\} \right)^*.
\end{equation}

Since the symbols $w$ never change their position, never appear or disappear, and the heads cannot ``jump over" them, then any transformation sequence starting with $Z$ can be presented as a chain of independent sequences which start with words $T_1w, wT_2w, \dots, wT_k$. Therefore
\begin{equation}
    \gamma(Z) = \gamma(T_1 w) + \gamma(w T_2 w) + \dots + \gamma(w T_k).
\end{equation}

Lemmata \ref{lem: one head full upper} and \ref{lem: many heads upper linear} imply that for all $V_i = w T_i w\ (i \not\in\{1, k\})$
\[
\gamma(V_i) \leq 4|V_i| \log |V_i| + 10|V_i|.
\]
As for $V_1 = T_1 w$ and $V_k = w T_k$, Lemma \ref{lem: upper one wall} and Remark \ref{rem: many heads} imply $\gamma(V_j) \leq 4|V_j|, \ j \in \{1, k\}$.

Therefore,
\begin{equation}
\label{eq: final gamma}
\gamma(Z) \leq 4 \sum\limits_{i = 2}^{k-1} \left( |V_i| \log |V_i| \right) + 10 \sum\limits_{i=2}^{k-1} |V_i| + 4\left(|V_1| + |V_k| \right).
\end{equation}

Notice that
\[
|Z| = |T_1| + 1 + \sum\limits_{i = 2}^{k-1} \left( |T_i| + 1\right) + |T_k| = |T_1| + |T_k| + \sum\limits_{i = 2}^{k-1}|T_i| + k - 1
\]
and also
\[
\sum\limits_{i=2}^{k-1} |V_i| = \left( \sum\limits_{i=2}^{k-1} |T_i| \right) + 2k - 4 = |Z| + k - 3 - |T_1| - |T_k|.
\]
Considering that $k \leq |Z| + 1$ (equality holds if $Z = w$), we obtain
\[
\sum\limits_{i=2}^{k-1} |V_i| \leq |Z| + |Z| - 1 - |T_1| - |T_k|  \leq 2|Z|
\]
Also, $|V_1| = |T_1| + 1 \leq |Z|$ and $|V_k| = |T_k| + 1 \leq |Z|$.

Now notice that the function $f(n) = n \log n$ is superadditive, i.e. for all $a,b$ holds $f(a) + f(b) \leq f(a+b)$, therefore
\[
\sum\limits_{i = 2}^{k-1} \left( |V_i| \log |V_i| \right) \leq
\left( \sum\limits_{i=2}^{k-1} |V_i| \right) \log \left( \sum\limits_{i=2}^{k-1} |V_i| \right) \leq
2 |Z| \log \left( 2|Z| \right).
\]
Collecting everything into \eqref{eq: final gamma}, we obtain
\[
\gamma(Z) \leq 4 \cdot 2|Z|\log\left(2|Z|\right) + 10 \cdot 2|Z| + 4\left(|Z| + |Z| \right) =
8 |Z|\log(2|Z|) + 28 |Z|
\]
thus proving a general upper estimate
\begin{lem}
For any word $Z \in A_0^*$
$$\gamma(Z) \leq 8 |Z| \log |Z| + 28 |Z|.$$
\end{lem}
\begin{cor}
\label{sec: Gamma upper}
    $\Gamma(n) \in O(n \log n)$.
\end{cor}

From \ref{cor: Gamma lower} and \ref{sec: Gamma upper} we conclude $\Gamma(n) \in \Theta(n\log n)$, whence, considering \eqref{eq: gamma to d} and \eqref{eq: dehn equals theta}, we get the statement of the Theorem \ref{th: main theorem}.

\section{Analysis of the construction method}
\label{sec: analysis}

The general idea for construction of the semigroup $S$ is simple: we take a semi-Thue system with a particular derivational complexity function, consider a semigroup with same alphabet and relation set, and it occurs that the semigroup's Dehn function (asymptotically) coincides with initial derivational complexity function. Such method would allow us to expand the spectrum of semigroup Dehn functions even further.

While it seems doubtful that the method works in general case, one may assume that it would work in case of semigroups with some ``good" qualities that $S_0$ has. Such qualities may be length-preservation and unambiguity (meaning that all the left parts of the rules are distinct). Here we will give an example of a rewriting system $E_0$, which has these qualities, however the method would not work. meaning that $\Delta_{E_0}(n)$ would not asymptotically coincide with the Dehn function $D_E(n)$ of the corresponding semigroup.

\begin{exmp}
The semi-Thue system $E_0 = (A, \cc{R})$ has an alphabet $A = \{0, \ol{0}, L, R, W\}$ and rewriting rules $\cc{R}$:
    \begin{align*}
        R0 &\ra \ol{0} R        &WL &\ra WR \\
        \ol{0} L &\ra L 0       &LWW &\ra R0W
    \end{align*}
    Symbols $0$ and $\ol{0}$ are called digits, $L$ and $R$ --- heads, $W$ --- walls.\\
    Further we will show that $\Delta_{E_0}(n) \in \Theta(n)$, however $D_E(n) \cgeq n^2$.
\end{exmp}

One can clearly see that $E_0$ is unambiguous and length-preserving. Moreover, it is a lot like $S_0$, in particular, every left and right part of the rules have exactly one head symbol, hence we also can talk about applying rules to the heads. Heads in a same way ``move" along the word through digits: $L$ moves left, $R$ moves right. Walls stay in place and (those, which are in the starting word or appear later, do not change their position or disappear) and also block the movement of the heads. Therefore, like in case of the system $S_0$ and the semigroup $S$, we can consider the segments between walls independently. The heads also cannot ``jump over" each other.

Notice that during direct transformations the head $R$ cannot become $L$, hence any number of heads in a word $v = WTW,\ T\in (A\backslash \{W\})^*$ will finish its movement after not longer than linear (of the word's length) number of steps. A rough estimate: each head can move not more than full length of $T$ to the left, then turn into $R$ and move full $|T|$ to the right --- overall, not more than $2|v|$ steps. If there are two or more heads, they bump into each other and the number of steps would be even less. Therefore, $\Delta_{E_0}(n) \in \Theta(n)$.

Consider the semigroup $E = \lan\lan A \mid \cc{R} \ran\ran$. After symmetrization of rules, a transformation $R0W~\rar~LWW$ ``unlocks", allowing the head $R$ to change its type and thus move in cycles.

Let us look at $E$ a bit differently: consider a set of rewriting rules $\cc{R}'$:
\begin{align*}
        R0 &\ra \ol{0} R        &WL &\ra WR \\
        \ol{0} L &\ra L 0       &R0W &\ra LWW
\end{align*}
It differs from $\cc{R}$ in a way, that the transformation $R0W \ra LWW$ is now \emph{direct}. However, the symmetrized rule set stays the same, i.e. $\lan\lan A \mid \cc{R}'\ran\ran \cong S$. Essentially, we just renamed a pair of transformations: the reverse one became direct, and the direct one became reverse.

\medskip

Consider a word $u = WR0^kW$. For it, there is a sequence of direct transformations, comprised of cycles. The first one looks like this:
\[
WR0^kW \ra^{k-1} W0^{k-1}R0W \ra W0^{k-1}LWW \ra^{k-1} WL0^{k-1}WW \ra WR0^{k-1}WW.
\]
The length of the cycle equals $2k$ steps, and one 0 digit turned into a wall. After $k$ cycles without the last transformation, the word would look like $v = WLW^{k+1}$, the length of each cycle is $2m$ steps, where $m$ equals the number of 0 digits at the beginning of the cycle. Overall number of steps would be
\[
\sum\limits_{i = k}^1 2i -1 = 2 \frac{k(k+1)}{2} - 1 = k(k+1) - 1 .
\]

Meanwhile, the reverse transformations do not decrease the number of 0 digits in the word, and the last transformation in the sequence was $WR0W^{k} \rad WLW^k$, thus, by analogy to the proof of Lemma \ref{lem: S lower}, the sequence above is the shortest one, connecting the words $u$ and $v$, therefore $d(u,v) = k(k+1) - 1$, which implies that $D_E(n) \geq k(k+1) - 1$, where $k = \floor{\frac{n}{2}} - 3$, since $|u| + |v| = 2(k + 3)$.

This implies that $D_E(n) \cgeq n^2$, therefore $D_E(n)$ definitely does not asymptotically coincide with $\Delta_{E_0}(n) \in \Theta(n)$.

\section{Discussion and further questions}

As shown in Section \ref{sec: analysis}, the applied technique does not provide the desired result in general case. Moreover, the presented example system $E_0$ has \emph{explicit} qualities: length-preservation and unambiguity. Other than these, the system $E_0$ does not seem to have any notable explicit qualities. This might lead to a conclusion that, if the method does work for some case of semi-Thue systems, such case may need to have some special implicit structure.

One example of such a special case might be a group of systems $S_{a,b}$ introduced in the paper \cite{AlTal}. These systems have the derivational length $\Delta_{a,b}$ (for each pair of natural $a$ and $b$) in the class $\Theta(n^{1 + \frac{a}{b}})$.
\begin{quest}
    Is it true that, for semi-Thue systems $S_{a,b}$ defined in \cite{AlTal}, the Dehn functions of the semigroups obtained by symmetrizing the rules of the systems are equivalent to $n^{1+\frac{a}{b}}$?
\end{quest}

\begin{quest}
    Is there a class of semi-Thue systems defined by some set of \emph{explicit} qualities, for which the technique of symmetrizing the rules to obtain a semigroup provides Dehn functions equivalent to derivational complexities of the initial systems?
\end{quest}

\section{Acknowledgements}

This work was supported by BASIS foundation as a part of Junior Leader project ``Combinatorial and algorithmic properties of semigroups and semi-Thue systems".

The author would like to thank Alexey Talambutsa for the assistance in writing and polishing the paper.

\printbibliography

@inproceedings{AlTal,
  title={On subquadratic derivational complexity of semi-thue systems},
  author={Talambutsa, Alexey},
  booktitle={International Computer Science Symposium in Russia},
  pages={379--392},
  year={2020},
  organization={Springer}
}

@article{birget,
  title={Time-complexity of the word problem for semigroups and the Higman embedding theorem},
  author={Birget, Jean-Camille},
  journal={International journal of algebra and computation},
  volume={8},
  number={02},
  pages={235--294},
  year={1998},
  publisher={World Scientific}
}

@article{sapir-birget-rips,
  title={Isoperimetric and isodiametric functions of groups},
  author={Sapir, Mark V. and Birget, Jean-Camille and Rips, Eliyahu},
  journal={Annals of Mathematics},
  pages={345--466},
  year={2002},
  publisher={JSTOR}
}

@article{olshanskii,
  title={Polynomially-bounded Dehn functions of groups},
  author={Olshanskii, Alexander Yu.},
  journal={Journal of Combinatorial Algebra},
  volume={2},
  number={4},
  pages={311--433},
  year={2018}
}

@article{madlener1985pseudo,
  title={Pseudo-natural algorithms for the word problem for finitely presented monoids and groups},
  author={Madlener, Klaus and Otto, Friedrich},
  journal={Journal of Symbolic Computation},
  volume={1},
  number={4},
  pages={383--418},
  year={1985},
  publisher={Academic Press, Inc. Orlando, FL, USA}
}

@article{pride1995geometric,
  title={Geometric methods in combinatorial semigroup theory},
  author={Pride, Stephen J},
  journal={NATO ASI Series C Mathematical and Physical Sciences-Advanced Study Institute},
  volume={466},
  pages={215--232},
  year={1995},
  publisher={Dordrecht, Holland, Boston, D. Reidel Pub. Co., l973-}
}

@inproceedings{duncan2004word,
  title={Word hyperbolic semigroups},
  author={Duncan, Andrew and Gilman, Robert H.},
  booktitle={Mathematical Proceedings of the Cambridge Philosophical Society},
  volume={136},
  pages={513--524},
  year={2004},
  organization={Cambridge University Press}
}

@article{baumslag1993isoperimetric,
  title={Isoperimetric inequalities and the homology of groups},
  author={Baumslag, Gilbert and Miller, Charles F. and Short, Hamish},
  journal={Inventiones mathematicae},
  volume={113},
  number={1},
  pages={531--560},
  year={1993}
}

@article{bridson1999fractional,
  title={Fractional isoperimetric inequalities and subgroup distortion},
  author={Bridson, Martin},
  journal={Journal of the American Mathematical Society},
  volume={12},
  number={4},
  pages={1103--1118},
  year={1999}
}

@article{gilman2002definition,
  title={On the definition of word hyperbolic groups},
  author={Gilman, Robert H.},
  journal={Mathematische Zeitschrift},
  volume={3},
  number={242},
  pages={529--541},
  year={2002}
}

@book{gromov1987hyperbolic,
  title={Hyperbolic groups},
  author={Gromov, Michael},
  year={1987},
  publisher={Springer}
}

@article{papasoglu1995sub,
  title={On the sub-quadratic isoperimetric inequality},
  author={Papasoglu, Panagiotis},
  journal={Geometric Group Theory. Boston: Jones and Bartlett},
  pages={149--158},
  year={1995}
}

@article{yu1991hyperbolicity,
  title={Hyperbolicity of groups with subquadratic isoperimetric inequality},
  author={Olshanskii, Alexander Yu.},
  journal={International Journal of Algebra and Computation},
  volume={1},
  number={03},
  pages={281--289},
  year={1991},
  publisher={World Scientific}
}

@article{cain2013hyperbolicity,
  title={Hyperbolicity of monoids presented by confluent monadic rewriting systems},
  author={Cain, Alan J},
  journal={Beitr{\"a}ge zur Algebra und Geometrie/Contributions to Algebra and Geometry},
  volume={54},
  pages={593--608},
  year={2013},
  publisher={Springer}
}

@article{nyberg2022dehn,
  title={On the Dehn functions of a class of monadic one-relation monoids},
  author={Nyberg-Brodda, Carl-Fredrik},
  journal={arXiv preprint arXiv:2210.16123},
  year={2022}
}

\end{document}